\documentclass[12pt,reqno]{article}
\usepackage{amsmath,amsthm,amsfonts,amssymb,amscd}
\usepackage[pdftex]{graphicx}
\usepackage{psfrag}
\usepackage{textcomp}
\usepackage[pctex32]{color}
\usepackage[pagewise]{lineno}

\theoremstyle{plain}
\newtheorem{thm}{Theorem}[section]
\newtheorem{rk}[thm]{Remark}
\newtheorem{conj}[thm]{Conjecture}
\newtheorem{prop}[thm]{Proposition}
\newtheorem{clly}[thm]{Corollary}
\newtheorem{lemma}[thm]{Lemma}
\newtheorem{defi}[thm]{Definition}

\newtheorem{maintheorem}{Theorem}

\newcommand{\re}{{\Bbb R}}
\newcommand{\nat}{{\Bbb N}}

\title{On the intersection of homoclinic classes in intransitive sectional-Anosov flows}
\author{H. M. S\'anchez
        \thanks{
This work is partially supported by CAPES, Brazil.}}
\date{}

\begin{document}
\maketitle

\begin{abstract}
We show that if $X$ is a {\em Venice mask}\index{Venice mask}
(i.e. nontransitive sectional-Anosov flow with dense periodic orbits,
\cite{bmp}, \cite{mp}, \cite{mp2},\cite{ls})
supported on a compact $3$-manifold, then the omega-limit set of  
every non-recurrent point in the unstable manifold of some singularity is a closed orbit. In addition, we prove that 
the intersection of two different homoclinic classes in the maximal invariant set of a sectional-Anosov flow can be 
decomposed as the disjoint union of, singular points, a non-singular hyperbolic set, and regular points whose 
{\em alpha-limit set} and {\em omega-limit set} is formed by singular points or hyperbolic sets.
\end{abstract}


\section{Introduction}

The dynamical systems theory is interested to describes the behavior as time goes
to infinity for the majority of orbits in a determinated system. An important tool for hyperbolic sets \index{Hyperbolic
set} is the known {\em connecting lemma} \index{Connecting lemma} \cite{hay}, \cite{ap}, \cite{bdv}. Specifically,
the lemma says that if $X$ is an Anosov flow on a compact manifold $M$ and $p, q \in M$ satisfy that for all 
$\varepsilon> 0$ there is a trajectory from a point $\varepsilon$-close to $p$ to a point $\varepsilon$-close to $q$, 
then there is a point $x\in M $ such that $\alpha_X(x)=\alpha_X(p)$ and $\omega_X(x)=\omega_X(q)$.

In \cite{bm2} was proved a similar result for sectional-Anosov flows, which is known as {\em sectional-connecting lemma}.
\index{Sectional-connecting lemma}
Recall, the sectional hyperbolic sets and 
sectional Anosov flows were introduced in \cite{mo} and \cite{mem} respectively 
as a generalization of the hyperbolic sets and Anosov flows to include important examples such as 
the saddle-type hyperbolic attracting sets, the geometric and multidimensional 
Lorenz attractors \cite{abs}, \cite{bpv}, \cite{gw} and certain robustly transitive sets.
A fundamental hypothesis in the sectional-hyperbolic case consists in the alpha-limit set of $p\in M(X)$ to be 
non-singular. As the unstable manifold of every singularity $\sigma$ \index{Singularity} of a 
sectional-Anosov $X$ is 
contained in the maximal invariant set $M(X)$, would be interesting to know what is the omega-limit set of a point in
$W^u_X(\sigma)$. In fact, it can be seen as a extension of the {\em sectional-connecting lemma}.

On the other hand, the class of Venice masks (i.e. intransitive sectional-Anosov flows with dense periodic orbits)
has a particular interest since its existence shows that the spectral decomposition theorem \cite{sma}
is not valid in the sectional-hyperbolic case. Its study has been collected by different authors during the last years.
The examples exhibited in \cite{bmp}, \cite{ls}, \cite{mp} are characterized because the maximal invariant set can be
decomposed as the disjoint finite union of homoclinic classes. In addition, the intersection between two different
homoclinic classes is contained in the closure of the union of the unstable manifold of the singularities.
Specifically, this intersection can be decomposed  as the disjoint union of, a singularity
$\sigma$, a closed orbit $C$, and regular points such that its {\em alpha-limit set} \index{Alpha-limit set}
is $\sigma$ and the {\em omega-limit set} \index{Omega-limit set} is $C$.
Particularly, was proved in \cite{mp}, \cite{mp2} that every Venice mask with a unique singularity 
has these properties. 

In search of properties which allow to characterized the dynamic of Venice masks, \index{Venice mask} will be studied the 
behavior of homoclinic classes \index{Homoclinic classes} and its relation with the unstable manifolds of the singularities.

Let us state our results in a more precise way.\\

Consider a Riemannian compact manifold $M$ of dimension $n$ (a {\em compact $n$-manifold} for short).
$M$ is endowed with a Riemannian metric $\langle\cdot,\cdot\rangle$ and an
induced norm $\lVert\cdot\rVert$. 
		We denote by $\partial M$ the boundary of $M$.
		Let ${\cal X}^1(M)$ be the space
		of $C^1$ vector fields in $M$ endowed with the
		$C^1$ topology.
		Fix $X\in {\cal X}^1(M)$, inwardly
		transverse to the boundary $\partial M$ and denotes
		by $X_t$ the flow of $X$, $t\in I\!\! R$.
		
The {\em $\omega$-limit set} of $p\in M$ is the set
$\omega_X(p)$ formed by those $q\in M$ such that $q=\lim_{n \rightarrow \infty}X_{t_n}(p)$ for some
sequence $t_n\to\infty$. The {\em $\alpha$-limit set} of $p\in M$ is the set
$\alpha_X(p)$ formed by those $q\in M$ such that $q=\lim_{n \rightarrow \infty}X_{t_n}(p)$ for some
sequence $t_n\to -\infty$. The {\em non-wandering set} of $X$ is
the set $\Omega(X)$ of points $p\in M$ such that
for every neighborhood $U$ of $p$ and every $T>0$ there
is $t>T$ such that $X_t(U)\cap U\neq\emptyset$.
Given $\Lambda \in M$ compact, we say that $\Lambda$ is {\em invariant} 
if $X_t(\Lambda)=\Lambda$ for all $t\in I\!\! R$.
We also say that $\Lambda$ is {\em transitive} if
$\Lambda=\omega_X(p)$ for some $p\in \Lambda$; {\em singular} if it
contains a singularity and
{\em attracting}
if $\Lambda=\cap_{t>0}X_t(U)$
for some compact neighborhood $U$ of it.
This neighborhood is often called
{\em isolating block}.
It is well known that the isolating block $U$ can be chosen to be
positively invariant, i.e., $X_t(U)\subset U$ for all
$t>0$.
An {\em attractor} is a transitive attracting set.
An attractor is {\em nontrivial} if it is
not a closed orbit.

The {\em maximal invariant} set of $X$ is defined by
	$M(X)= \bigcap_{t \geq 0} X_t(M)$.

\begin{defi}\label{defhyp}
		\label{hyperbolic}
		A compact invariant set $\Lambda$ of $X$ is {\em hyperbolic}
		if there are a continuous tangent bundle invariant decomposition
		$T_{\Lambda}M=E^s\oplus E^X\oplus E^u$ and positive constants
		$C,\lambda$ such that

		\begin{itemize}
		\item $E^X$ is the vector field's
		direction over $\Lambda$.
		\item $E^s$ is {\em contracting}, i.e.,
		$
		\lVert DX_t(x) \left|_{E^s_x}\right.\rVert
		\leq Ce^{-\lambda t}$,
		for all $x \in \Lambda$ and $t>0$.
		\item $E^u$ is {\em expanding}, i.e.,
		$
		\lVert DX_{-t}(x) \left|_{E^u_x}\right.\rVert
		\leq Ce^{-\lambda t},
		$
		for all $x\in \Lambda$ and $t> 0$.
		\end{itemize}
\end{defi}
A compact invariant set $\Lambda$ has a {\em dominated splitting} with
respect to the tangent flow if there are an invariant splitting 
$T_{\Lambda}M = E\oplus F$ and positive numbers $K,\lambda$
such that

$$\lVert DX_t(x)e_x\rVert\cdot \lVert f_x\rVert\leq Ke^{-\lambda t} \lVert DX_t(x)f_x\rVert\cdot \lVert e_x\rVert,\qquad
\forall x\in\Lambda, t \geq 0, (e_x , f_x) \in E_x\times F_x .$$

Notice that this definition allows every compact invariant set $\Lambda$
to have a dominated splitting with respect to the tangent flow (See \cite{bamo}):
Just take $E_x = T_xM$ and $F_x = 0$,
for every $x\in\Lambda$ (or $E_x = 0$ and $F_x = T_x M$ for 
every $x\in\Lambda$). 
A compact invariant set $\Lambda$ is {\em partially hyperbolic} if
it has a {\em partially hyperbolic splitting}, i.e., a dominated splitting $T_{\Lambda}M = 
E\oplus F$ with respect to the tangent flow
whose dominated subbundle $E$ is contracting in the sense of
Definition \ref{defhyp}.

The Riemannian metric $\langle\cdot ,\cdot\rangle$ of $M$
induces a $2$-Riemannian metric \cite{mv},
$$\langle u, v/w\rangle_p= \langle u, v\rangle_p\cdot
\langle w, w\rangle_p - \langle u, w\rangle_p\cdot \langle v, w\rangle_p,\quad
\forall p\in M, \forall u, v, w \in T_p M.$$
This in turns induces a 2-norm \cite{gah} (or areal metric \cite{kat}) defined by
$$\lVert u, v\rVert =\sqrt{\langle u, u/v\rangle_p}
\qquad \forall p\in M, \forall u, v \in T_p M.$$ 

Geometrically, $\lVert u, v\rVert$ represents the area of the paralellogram 
generated by $u$ and $v$ in $T_p M$.

If a compact invariant set $\Lambda$ has a dominated splitting $T_{\Lambda}M = E \oplus F$ with
respect to the tangent flow, then we say that its central subbundle $F$ is {\em sectionally
expanding} if

$$\lVert DX_t (x)u, DX_t (x)v\rVert \geq K^{-1}e^{\lambda t}
\lVert u, v\rVert, \quad\forall x \in \Lambda, u,v\in F_x, t \geq 0.$$

By a {\em sectional-hyperbolic splitting} for $X$ over $\Lambda$ we mean a partially hyperbolic
splitting $T_{\Lambda}M = E\oplus F$ whose central subbundle $F$ is sectionally expanding.

\begin{defi}
A compact invariant set $\Lambda$ is {\em sectional-hyperbolic} for $X$ if 
its singularities are hyperbolic and if there is a sectional-hyperbolic splitting for $X$ over $\Lambda$.
\end{defi}

\begin{defi}
\label{secflow}
We say that $X$ is a {\em sectional-Anosov flow} if $M(X)$ is a sectional-hyperbolic
set.
\end{defi}

The Invariant Manifold Theorem [3] asserts that if $x$ belongs
to a hyperbolic set $H$ of $X$, then the sets

$$W^{ss}_X(p)  =  \{x\in M:d(X_t(x),X_t(p))\to 0, t\to \infty\} \qquad and$$
$$W^{uu}_X(p)  =  \{x\in M:d(X_t(x),X_t(p))\to 0, t\to -\infty\},$$

are $C^1$ immersed submanifolds of $M$ which are tangent at $p$ to the subspaces $E^s_p$ and $E^u_p$ of $T_pM$ respectively.

$$W^{s}_X(p) =   \bigcup_{t\in I\!\! R}W^{ss}_X(X_t(p))\qquad\and\qquad W^{u}_X(p)  =   \bigcup_{t\in I\!\! R}W^{uu}_X(X_t(p))$$

are also $C^1$ immersed submanifolds tangent to $E^s_p\oplus E^X_p$ and $E^X_p\oplus E^u_p$ at $p$ respectively.

Recall that a singularity of a vector field is hyperbolic if
the eigenvalues of its linear part
have non zero real part.

\begin{defi}
\label{ll}
We say that a singularity $\sigma$ of a sectional-Anosov flow $X$ is {\em Lorenz-like}
if  it has three real eigenvalues $\lambda^{ss},\lambda^{s},\lambda^u$ with $\lambda^{ss}<\lambda^s<0<-\lambda^s<\lambda^u$.
such that the real part of the remainder eigenvalues are outside the compact interval $[\lambda^{s},\lambda^u]$. 
$W^s_X(\sigma)$ is the manifold associated to the eigenvalues with negative real part. 
The strong stable foliation\index{Strong stable foliation} associated to $\sigma$ and denoted by
$\mathcal{F}^{ss}_X(\sigma)$, is the foliation contained in $W^s_X(\sigma)$ which is tangent
to space generated by the eigenvalues with real part less than $\lambda^{s}$.
\end{defi}

\begin{defi}
A periodic orbit of $X$ is the orbit of some $p$ for which there is a minimal
$t > 0$ (called the period) such that $X_t(p) = p$. An orbit is called closed if it is a periodic orbit
or a singularity.
\end{defi}

A homoclinic orbit of a hyperbolic periodic orbit $O$ is an orbit $\gamma\subset W^s(O)\cap W^u(O)$.
If additionally $T_qM = T_qW^s(O) + T_qW^u(O)$ for some (and hence all) point $q\in\gamma$, 
then we say that $\gamma$ is a transverse homoclinic orbit of $O$. 
The homoclinic class $H(O)$ of a hyperbolic periodic orbit $O$ is the closure of the 
union of the transverse homoclinic orbits of $O$. 
We say that a set $\Lambda$ is a homoclinic class if $\Lambda = H(O)$ for some hyperbolic periodic orbit $O$.

\begin{defi}
A Venice mask is a sectional-Anosov
flow with dense periodic
orbits which is not transitive.
\end{defi}

If $A$ is a compact invariant set of $X$ we denote $Sing_X(A)$ the set of
singularites of $X$ in $A$, and $Sing(X)=Sing_X(M(X))$. The closure of $B\subset M$
is denoted by $Cl(B)$. With these definitions we can state our main results.

\section{Main statements}

We show that if $X$ is a {\em Venice mask}\index{Venice mask}
supported on a compact $3$-manifold, then the omega-limit set of  
every non-recurrent point in the unstable manifold of some singularity is a closed orbit. In addition, we prove that 
the intersection of two different homoclinic classes in the maximal invariant set of a sectional-Anosov flow can be 
decomposed as the disjoint union of, singular points, a non-singular hyperbolic set, and regular points whose 
{\em alpha-limit set} and {\em omega-limit set} is formed by singular points or hyperbolic sets.\\


Specifically, we have the following statements.

\begin{maintheorem}
\label{thH}
If $X$ is a three-dimensional Venice mask and $\sigma$ is a singularity of $X$, then for every 
$q\in W^u_X(\sigma)$ such that $q$ is non-recurrrent we have the following dichotomy:
\begin{itemize}
\item $\omega_X(q)\in Sing(X)$.
\item $\omega_X(q)=O$, where $O$ is a hyperbolic periodic orbit.
\end{itemize}

\end{maintheorem}

\begin{maintheorem}
\label{thH'}
The intersection of two different homoclinic classes $H_1,H_2$ in the maximal invariant set of a sectional-Anosov flow
$X$ is the disjoint union of a set $S$ (possibly empty) of singularities, a non-singular hyperbolic set $H$ (possibly
empty), and a set $R$ (possibly empty) of regular points \index{Regular point} such that if $q\in R$ then 
$\alpha_X(q)\subset H\cup S$ and $\omega_X(q)\subset H\cup S$.
\end{maintheorem}


 \section{Preliminaries}
\label{prelim}

We mention the following results which are essentials to proving the theorems.

\begin{thm}[\cite{mpp2}]
\label{th1}
Let $\Lambda$ be a sectional-hyperbolic set with dense periodic orbits. \index{Dense periodic orbits} Then,
every $\sigma\in Sing_X(\Lambda)$ is Lorenz-like and satisfies $\Lambda\cap \mathcal{F}^{ss}_X(\sigma) =\{\sigma\}$.
\end{thm}
 
We observe that $W^s_X(\sigma)\setminus \mathcal{F}^{ss}_X(\sigma)$ is decomposed by two connected components 
\index{Connected components} $W^{s,+}_X(\sigma)$ and $W^{s,-}_X(\sigma)$ (see figure \ref{Ws+}). Hence for a 
Venice mask, a regular point in $M(X)$ contained in the stable manifold \index{Stable manifold} of some singularity
$\sigma$, necessarily is contained either $W^{s,+}_X(\sigma)$ or $W^{s,-}_X(\sigma)$. 

\begin{figure}
\label{Ws+}
\begin{center}
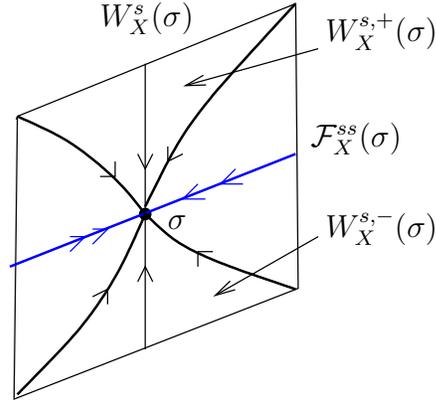
\caption{\label{VS} Connected components.}
\end{center}
\end{figure}

\begin{lemma}[Hyperbolic lemma \cite{mpp2}] 
\label{lem1}
A compact invariant set without singularities of a sectional-hyperbolic
set is hyperbolic saddle-type. \index{Hyperbolic lemma}
\end{lemma}

\begin{rk}
Theorem \ref{th1} and the Hyperbolic Lemma imply that every Venice mask has singularities, and these are
Lorenz-like.
 
\end{rk}

\begin{defi}
We say that a $C^1$ vector field $X$ with hyperbolic closed orbits
has the Property $(P)$ \index{Property $(P)$} if for every periodic orbit $O$ there is a singularity $\sigma$ such
that 
\begin{equation}
\label{P}
W^u_X(O)\cap W^s_X(\sigma)\neq \emptyset.
\end{equation}

\end{defi}

The above definition is useful by the interesting fact below.

\begin{lemma}
\label{lem2}
Every point in the closure of the periodic orbits of a vector field
with the Property $(P)$ is accumulated by points for which the omega-limit set is
a singularity. 
\end{lemma}

Moreover, we have an important property. 

\begin{lemma}[\cite{mp}]
\label{l3}
Every sectional-Anosov flow with singularities and dense periodic
orbits on a compact 3-manifold has the Property $(P)$.
\end{lemma}

\begin{rk}
By Lemma \ref{lem2} and Lemma \ref{l3} we can assert that every Venice mask $X$ has the Property $(P)$ and
$W^s(Sing(X))\cap M(X)$ is dense in $M(X)$.
 
\end{rk}

\begin{defi}
Given $\Sigma\subset M$ we say that $q\in M$ satisfies Property $(P)_{\Sigma}$ \index{Property $(P)_{\Sigma}$} if
$Cl(O^+(q))\cap\Sigma =\emptyset$ and there is open arc \index{Open arc} $I$ in $M$ with $q\in\partial I$ such that
$O^+(x)\cap\Sigma \neq\emptyset$ for every $x\in I$.
\end{defi}

We finish to exhibit the preliminar statements with the following characterization. 
\begin{thm}[\cite{bamo2}]
\label{thPsigma}
Let $X$ be a $C^1$ vector field in a compact 3-manifold $M$. If $q\in M$ 
has sectional-hyperbolic omega-limit set $\omega(q)$, then the following properties are
equivalent:

\begin{itemize}
\item $\omega(q)$ is a closed orbit\index{Closed orbit}. 
\item $q$ satisfies $(P)_{\Sigma}$ for some closed subset $\Sigma$. 
\end{itemize}
\end{thm}

In Figure \ref{PS} is exhibited the case when the omega-limit set
$\omega(q)$ of the point $q$ is a hyperbolic singularity of saddle-type.  

\begin{figure}
\begin{center}
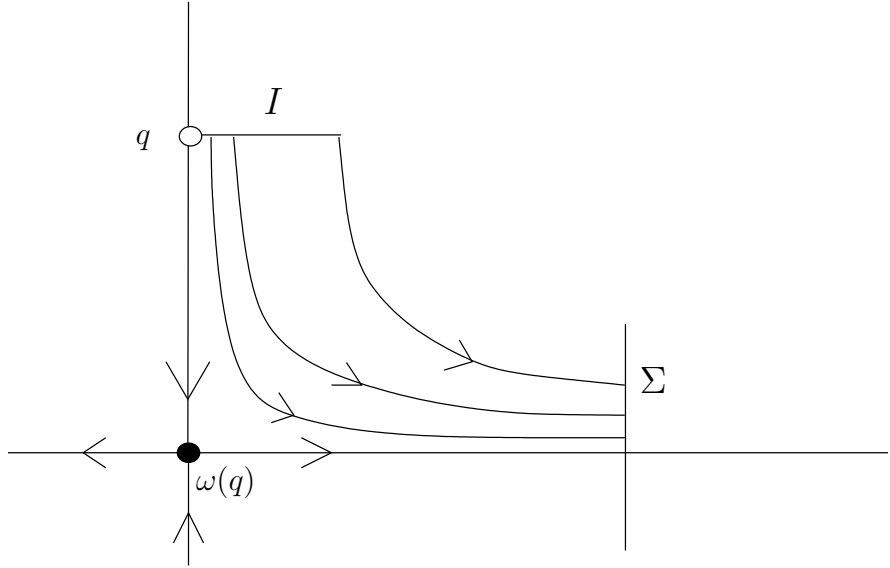
\caption{\label{PS} Property $(P)_{\Sigma}$.}
\end{center}
\end{figure}


\section{Characterizing the omega-limit set}
\label{omegalimit}

In this section we will prove the {\em Theorem \ref{thH}}. The idea is to consider a sequence of points 
satisfying the Property $(P)_{\Sigma}$, which approximates a point $q$ in the unstable manifold of a fixed
singularity. We show that $q$ satisfies the Property $(P)_{\Sigma}$ too. Hereafter in this section, we assume that every
regular point $q\in W^u(Sing(X))$ is non-recurrent.\index{Non-recurrent point}\\

First, we mention some facts of topology. Given a compact metric space $(Y,d)$, define a distance 
function between any point $x$ of $Y$ and any non-empty set $B$ of $Y$ by:

$$d(x,B)=\inf \{ d(x,y) | y \in B \}.$$

Now, consider the collection $\mathcal{C}(Y)=\{C\in Y: C \text{ is a non-empty compact subset of $(Y,d)$}\}$.
For $\mathcal{C}(Y)$, take the Hausdorff metric $d_H$ defined as the distance function between any two 
non-empty sets $A$ and $B$ of $Y$ by:

$$d_H(A,B)=\sup \{ d(x,B) | x \in A \}.$$

\begin{lemma}
\label{l4}
Let $\{A_n:n\in\nat\}$ be a sequence of closed sets contained in a compact metric space $(Y,d)$, such that $A_n\to A$
in the Hausdorff metric induced by $d$. Then $\partial A_n\to \partial A$.
\end{lemma}

For now and on this section, let $M$ be a riemaniann compact 3-manifold, and let $X$ be a Venice mask on $M$. 
So, for a hyperbolic point $p$ of $X$, $W^s_X(p)$ is just denoted by $W^s(p)$. The same interchanging $s$ by $u$.\\

\subsection{Existence of singular partitions}
\label{singpart}

We introduce the following definition which can also be found
in \cite{bau} and \cite{lec}, and extends the notion given in \cite{mp1}. \\

A cross section \index{Cross section} of $X$ is a codimension one submanifold $S$ transverse to $X$. We denote the
interior and the boundary (in topological sense) of $S$ by $Int(S)$ and $\partial S$ respectively. If $\mathcal{R} =
\{S_1 ,\cdots , S_k \}$ is a collection of cross sections we still denote by $\mathcal{R}$ the union of
its elements. Moreover

$$\partial \mathcal{R} :=\bigcup_{i=1}^k\partial S_i\qquad and \qquad Int(\mathcal{R}) :=\bigcup_{i=1}^k Int(S_i)$$ 

The size of $\mathcal{R}$ will be the sum of the diameters of its elements.

\begin{defi}
A singular partition \index{Singular partition} of an invariant set $H$ of a vector field $X$ is a finite disjoint
collection $\mathcal{R}$ of cross sections of $X$ such that $H\cap\partial\mathcal{R}=\emptyset$ and

$$
H\cap Sing(X)=\{y\in H: X_t(y)\notin \mathcal{R}, \forall t\in\re\}.
$$
\end{defi}

For a Lorenz-like singularity \index{Lorenz-like singularity}
$\sigma$, the center unstable manifold \index{Center unstable manifold}
$W_X^{cu}(\sigma)$ associated is divided by $W^u(\sigma)$ and $W^s(\sigma)\cap W^{cu}(\sigma)$ in the four sectors
$s_{11}$, $s_{12}$, $s_{21}$, $s_{22}$. $\pi: V_{\sigma}\to W^{cu}(\sigma)$ is the projection
defined in a neighborhood $V_{\sigma}$ of $\sigma$.  Figure \ref{cusigma} exhibits the case when $\pi(M(X)\cap V_{\sigma})$ 
intersects $s_{11}$ and $s_{12}$.

\begin{figure}
\begin{center}
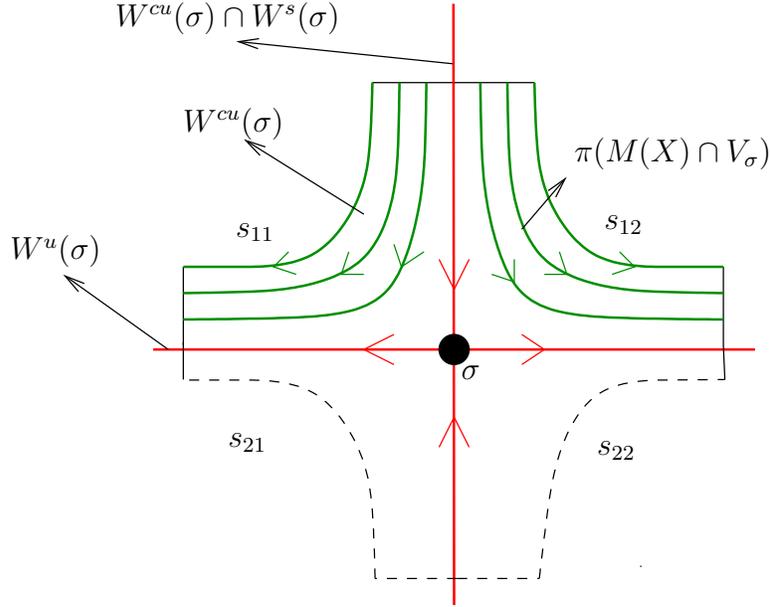
\caption{\label{cusigma} Center unstable manifold of $\sigma$.}
\end{center}
\end{figure}


\begin{lemma}
\label{lem41}
Consider $\sigma$ a Lorenz-like singularity of a Venice mask $X$, and $O$ a hyperbolic periodic orbit satisfying 
$Cl(W^u(O))\cap W^{s,+}(\sigma)\neq\emptyset$ and $Cl(W^u(O))\cap W^{s,-}(\sigma)\neq\emptyset$. Moreover, $\pi(Cl(W^u(O)))
\cap s_{1i}\neq\emptyset$ and $\pi(Cl(W^u(O)))\cap s_{2i}\neq\emptyset$ for some $i\in\{1,2\}$. If $q$ is a regular point
in $W^u(\sigma)\cap Cl(s_{1i})\cap Cl(s_{2i})$, then $O=\omega_X(q)$. 
\end{lemma}

\begin{proof}
We take $q\in W^u(\sigma)$ a regular point close to $\sigma$. We assert that $q\in W^s(O)$. Indeed, if we suppose 
that is not the case, we will get a contradiction. \\

So, we assume $q\in W^u(\sigma)\setminus W^s(O)$. Then, there is a  sequence $p_n^-\to q$ such that
$p_n^-\in W^u(O)$ for all $n$. 
In addition, $\{O_X(p_n^-):n\in\nat\}$ accumulates some regular point $p^-$ in $W^{s,-}(\sigma)$ or in
$W^{s,+}(\sigma)$. We can suppose the accumulation in some point of $W^{s,-}(\sigma)$.
Also, we can take $\{p_n^+:n\in\nat\}\subset W^u(O)$ be a sequence such that $p_n^+\to q$.
Moreover, $\{O_X(p_n^+):n\in\nat\}$ accumulates 
$\sigma$ and some point $p^+$ in $W^{s,+}(\sigma)$.  
We have $p_n^+,p_n^-\notin W^u(\sigma)$ for all $n$. On the other hand, $q\in Cl(W^u(O))$ and the invariance of 
$W^u(\sigma)$ imply $O_X(q)\subset Cl(W^u(O))$. But $Cl(W^u(O))$ is a closed set, therefore $Cl(O_X(q))\subset Cl(W^u(O))$.   
Applying the compactness of $Cl(W^u(O))$ and Tubular Flow Box Theorem \index{Tubular Flow Box Theorem}
\cite{dmp} in a neighborhood of $O^+(q)$ we obtain that $\{O^+(p_n^+):n\in\nat\}$ and $\{O^+(p_n^-)n\in\nat\}$ 
accummulate all point in $W^u(\sigma)$ close to $\omega_X(q)$. \\

As $O$ and $\omega_X(q)$ are invariant closed 
sets, then they are disjoints and $d(x,\omega_X(q))>0$ for all $x\in O$. This implies that there exists $\varepsilon>0$
such that every point $y$ closen
to $\omega_X(q)$  satisfies $d(y,O)>\varepsilon$. Moreover $y\notin O_X(q)$ and, $\{O^+(p_n^+):n\in\nat\}$, 
$\{O^+_X(p_n^-):n\in\nat\}$ acummulate $y$. The positive
orbits of $p_n^+$ and $p_n^-$ cannot intersect $\omega_X(q)$. So, we have two possibilities, either any orbit intersects
$O_X(q)$, or no orbit does it. The first case means that there is a point $w\in W^u(\sigma)\cap W^u(O)$ which is absurd.  
So, neither orbit intersects $O_X(q)$. Now, $q$ is a non-recurrent point. Then, $\{O^+_X(p_n^+):n\in\nat\}$ does not
accumulate on $W^{s,+}(\sigma)$. But this contradicts the 
choice of the sequences. Therefore $q\in W^s(O)$. So, we conclude $O=\omega_X(q)$.

\end{proof}

From {\em Lemma} \ref{lem41} we obtain the following corollary.\\

\begin{clly}
\label{clly5} 
Consider $\sigma$ a Lorenz-like singularity of a Venice mask $X$, and $O$ a hyperbolic periodic orbit satisfying 
$W^u(O)\cap W^{s,+}(\sigma)\neq\emptyset$ and $W^u(O)\cap W^{s,-}(\sigma)\neq\emptyset$. Let $q$ be a regular point in 
$W^u(\sigma)\cap Cl(W^u(O))$ and let $\{p_n:n\in\nat\}\subset Cl(W^u(O))\cap W^s(O)$ be a sequence such that $p_n\to q$.
Then $p_n\in O_X(q)$ for all $n$ large.
\end{clly}

\begin{proof}
For this is sufficient to observe that $O_X(q)$ is contained in $W^s(O)$.

\end{proof}

\begin{rk}
\label{rk3}
{\em Corollary} \ref{clly5} says that for $i\in\{1,2\}$ and for every hyperbolic periodic orbit $O$ of $X$, is not possible
$H(O)\cap s_{1i}\neq\emptyset$ and $H(O)\cap s_{2i}\neq\emptyset$ simultaneously. 
\end{rk}

\begin{lemma}
\label{lem6}
Let  $\sigma$ be a singularity of a Venice mask $X$, and let $O$ be a hyperbolic periodic
orbit such that $W^u(O)\cap W^s(\sigma)\neq\emptyset$. Then for $q\in W^u(\sigma)\setminus\{\sigma\}$,
$\omega_X(q)$ has singular partitions \index{Singular partition} of arbitrarily small size.
\end{lemma}

\begin{proof}
We adapt the proof of {\em Theorem 17} given in \cite{lec}. Observe that $\omega_X(q)$ is 
sectional-hyperbolic. Therefore, if $\omega_X(q)$ is a closed orbit, then {\em Theorem}
\ref{thPsigma} implies that $q$ satisfies the property $(P)_{\Sigma}$ \index{Property $(P)_{\Sigma}$} for some closed 
subset $\Sigma$. Moreover, we can apply {\em Theorem 16} in \cite{lec} to conclude that $\omega_X(q)$ has singular 
partitions of arbitrarily small size.\\

Hereafter, we assume $\omega_X(q)$ is not a closed orbit. By {\em Proposition 3} in \cite{lec} is sufficient to prove that
for all $z\in \omega_X(q)$ there is cross section $\Sigma_z$  close to $z$ such that $z\in Int(\Sigma_z)$ and
$\omega_X(q)\cap\partial\Sigma_z= \emptyset$.\\

We assert that $\omega_X(q)$ cannot contain any local strong stable manifold. \index{Strong stable manifold}
Indeed, we first assume that $\omega_X(q)$ has no singularities. By {\em Hyperbolic lemma}, \index{Hyperbolic lemma}
it is hyperbolic saddle-type. Suppose $\omega_X(q)$ containing a local 
strong stable manifold. Then, by {\em Lemma 11}  in \cite{lec}, $q$ would be a recurrent point. Therefore using 
{\em Lemma 5.6} in \cite{mpa1}, there is $x^*\in Per(X)\cap \omega_X(q)$ such that $q\in W^s_X(x^*)$. This means that 
$\omega_X(q)$ is a periodic orbit  which contradicts our assumption. Now, if $\omega_X(q)$ is a sectional-hyperbolic set
with singularities, applying {\em Main Theorem} in \cite{mor}, $\omega_X(q)$ cannot contain any local strong stable 
manifold.\\

We can fix a foliated rectangle of small diameter $R_z^0$ such that $z\in Int(R_z^0)$ and $\omega_X(q)\cap\partial^hR_0^z
=\emptyset$. By {\em Theorem} \ref{th1}, the intersection of $W^u(O)$ with $W^s(\sigma)$ occurs in some connected component 
$W^{s,+}(\sigma)$ or $W^{s,-}(\sigma)$ (or both). We initially assume the intersection in $W^{s.+}(\sigma)$.\\

Since $z\in\omega_X(q)$ and the omega-limit set is not a closed orbit, we have that 
the positive orbit of $q$ intersects either only one or the two connected components
of $R_z^0\setminus \mathcal{F}^s(z,R^0_z)$. \\

Assume the intersection is occurring in just one component only, we shall consider the following cases:\\

\begin{itemize}
 \item $W^{s,-}(\sigma)\cap M(X)=\emptyset$. 
 
Using this and linear coordinates around $\sigma$, we can construct an open interval $I^+ = I^+_q\subset W^u(O)$,
contained in a suitable cross section throught $q\in W^u(\sigma)\setminus \{\sigma\}$ and $q\in\partial I^+$. 
As $W^u(O)\cap W^{s,+}(\sigma)$ is dense in $W^{u}(O)$ we have $I^+\cap W^{s,+}(\sigma)$ is dense in 
$I^+$.
 
It is possible to assume $I^+$ is contained in that component of $R_z^0\setminus \mathcal{F}^s(z,R^0_z)$. It is because of
the positive orbit of $q$ carries the positive orbit of $I^+$ into such a component. Furthermore, the stable
manifolds throught $I^+$ form a subrectangle $R_I^+$ in there. So, $W^{s,+}(\sigma)\cap R_I^+$ is dense in $R_I^+$.

Now, as in  {\em Theorem 17} of \cite{lec},  we suppose $\omega_X(q)\cap Int(R_I^+)\neq\emptyset$ to obtain a 
contradiction. By hypothesis, the omega-limit set of $q$ is not a periodic orbit. Then {\em Lemma 5.6} in \cite{mpa1} 
implies that the positive orbit of $q$ cannot intersects $\mathcal{F}^s(q,R^0_z)$ infinitely many times. Now, if it
intersects $R_I^+$, then by the density of $W^{s,+}(\sigma)\cap R_I^+$ in $R_I^+$, we can assert that the positive orbit 
of a point $p$ in $W^{s,-}(\sigma)$ would intersect $R_I^+$. Therefore $p\in Cl(W^u(O))\subset M(X)$ which we get a 
contradiction. So $\omega_X(q)\cap Int(R_I^+)=\emptyset$.
 
To continue, we choose a point $z'\in Int(R_I^+)$ and a point $z''$ in the connected component $R_z^0\setminus
\mathcal{F}^s (z, R_z^0 )$ not intersected by the 
positive orbit of $q$.	The desired rectangle $\Sigma_z$ is a subrectangle of $R^0_z$ bounded by $\mathcal{F}^s(z',R_z^0)$
and $\mathcal{F}^s(z'',R_z^0)$.\\

\item  $W^s(\sigma)\cap W^u(O)\subset W^{s,+}(\sigma)$ and $W^s(\sigma)\cap W^u(O')\subset W^{s,-}(\sigma)$
for some hyperbolic periodic orbit $O'\neq O$.

In this way, we have the hypotheses of {\em Theorem 17} in \cite{lec}. Therefore
there exists an interval $I^-\subset W^u(O')$ contained in that component of $R_z^0\setminus \mathcal{F}^s(z,R^0_z)$,
such that $q\in\partial I^-$ and $I^-\cap W^{s,-}(\sigma)$ is dense in $I^-$. The
stable manifolds throught $I=I^+\cup\{q\}\cup I^-$ form a subrectangle $R_I$ in there, with $Int(R_I)\cap\omega_X(q)=
\emptyset$. So, the existence of $\Sigma_z$ is guaranteed such as last item. \\

\item $W^{s,+}(\sigma)\cap W^u(O)\neq\emptyset$ and $W^{s,-}(\sigma)\cap W^u(O)\neq\emptyset$.

We assert that there are $O_1,O_2$ hyperbolic periodic orbits such that, 
$W^s(\sigma)\cap W^u(O_1)\subset W^{s,+}(\sigma)$ and $W^s(\sigma)\cap W^u(O_2)\subset W^{s,-}(\sigma)$. Indeed, we take 
$q_1\in W^{s,+}(\sigma)\cap W^u(O)$ and $q_2\in W^{s,-}(\sigma)\cap W^u(O)$.

As $M(X)$ is union of homoclinic classes and $W^u(O)\subset M(X)$, there are hyperbolic periodic orbits $O_1,O_2$ satisfying
$q,q_1\in H(O_1)$ and $q,q_2\in H(O_2)$. Therefore $O_X(q_1)\subset H(O_1)$ and $O_X(q_2)\subset H(O_2)$. Moreover, since 
the homoclinic classes are closed set we have that $\sigma$ and $O$ are in $H(O_1)\cap H(O_2)$. From {\em Remark} \ref{rk3}
follows $H(O_1)\cap W^s(\sigma)\subset W^{s,+}(\sigma)$ and $H(O_2)\cap W^s(\sigma)\subset W^{s,-}(\sigma)$. 
On the other hand, let $W^+(O)$ be the connected component of $W^u(O)\setminus O$ containing $q_1$, then $W^+(O)
\subset H(O_1)$. Analogously, for $W^-(O)$, the connected component of $W^u(O)\setminus O$ containing $q_2$, we have
$W^-(O)\subset H(O_2)$. Therefore $W^u(O_1)\cap W^s(\sigma)\subset W^{s,+}(\sigma)$ and $W^u(O_2)\cap W^s(\sigma)
\subset W^{s,-}(\sigma)$. Again we have the hypotheses of {\em Theorem 17} in \cite{lec}.\\

\item $W^{s,+}(\sigma)\cap W^u(O)\neq\emptyset$ and $W^{s,-}(\sigma)\cap H(O)\neq\emptyset$.

It is not possible by {\em Corollary} \ref{clly5}.

\item $W^{s,+}(\sigma)\cap W^u(O)\neq\emptyset$, $W^{s,-}(\sigma)\cap Cl(W^u(O'))\neq\emptyset$ and $q\in Cl(W^u(O'))$, 
where $O'$ is a hyperbolic periodic orbit of $X$.

From last item $O'\notin H(O)$. As $X$ satisfies the Property $(P)$\index{Property $(P)$}, there is $\sigma'\in Sing(X)$
such that $W^u(O')\cap W^s(\sigma')\neq\emptyset$. If 
$\sigma'=\sigma$ then $W^u(O')$ intersects $W^{s,+}(\sigma)$ or $W^{s,-}(\sigma)$. Observe that those alternatives were 
already analyzed. If $\sigma'\neq\sigma$, then we can obtain an interval $J^-$ such that $J^-\subset W^u(O')$ and 
$J^-\cap W^s(\sigma')$ is dense in $J^-$. Moreover we can assume $W^{s}(\sigma)\cap W^u(O)\subset W^{s,+}(\sigma)$
to obtain an interval $I^+$ such that $I^+\subset W^u(O)$ and 
$I^+\cap W^{s,+}(\sigma)$ is dense in $I^+$. Since $O'\notin H(O)$, follows that $W^u(O')\nsubseteq H(O)$. 
Therefore $W^u(O')$ cannot intersect $W^{s,+}(\sigma)$. In this way, there is an open arc $I^-\subset 
\bigcup_{t\geq 0}X_t(J^-)$ such that $q\in\partial I^-$. $I^-$ works such as in second item. The stable manifolds 
\index{Stable manifold} throught
$I=I^+\cup\{q\}\cup I^-$ generates a subrectangle $R_I$. This acts such as {\em Theorem} 17 in \cite{lec}. 

\end{itemize}

Now assume the positive orbit intersects both components of $R_z^0\setminus\mathcal{F}^s(z,R_z^0)$. Therefore we take
$I$ (or $I^+$ to first case) with the positive orbit as before to obtain two subrectangles $R_I^t$ and $R_I^b$,
like $R_I$ (or $R_I^+$ to first case), in each component. Then we select two points $z'\in Int(R_I^t )$ and
$z''\in Int(R_I^b)$ and define $\Sigma_z$ as the rectangle in $R_z^0$ bounded by $\mathcal{F}^s(z', R_z^0)$ and
$\mathcal{F}^s(z'', R_z^0)$. 

From {\em Proposition} 3 in \cite{lec} we conclude the result.

\end{proof}

We remember the concept of {\em singular cross section} \index{Singular cross section} that appears in \cite{mp2}.
For a disjoint collection of rectangles $\mathcal{S}=\{S_1,\cdots,S_l\}$ we denote 
$\mathcal{S}^o=\mathcal{S}\setminus\partial\mathcal{S}$. and $\partial^{\ast}\mathcal{S}=\bigcup_{S\in\mathcal{S}}
\partial^{\ast} S$ for $\ast=h,v,o$.

\begin{defi}
A singular cross section of $X$ is a finite disjoint collection $\mathcal{S}$ of
foliated rectangles with $M(X)\cap\partial ^h S =\emptyset$ such that for every $S\in\mathcal{S}$ there is a
leaf $l_S$ of $\mathcal{F}^s$ in $S^o$ such that the return time $t_S(x)$ for $x\in S\cap Dom(\Pi_S )$ goes
uniformly to infinity as $x$ approaches $l_S$. 

We define the singular curve \index{Singular cruve} of $\mathcal{S}$ as the union,

$$l_{\mathcal{S}}=\bigcup_{S\in\mathcal{S}}l_S.$$ 

\end{defi}

\begin{prop}
\label{prop31}
Let $q$ be a regular point \index{Regular point} in $W^u(\sigma)$, with $\sigma$ a singularity of a Venice mask $X$, 
and let $O$ be a hyperbolic periodic orbit such that $W^u(O)\cap W^s(\sigma)\neq\emptyset$. Then $\omega_X(q)$
\index{Omega-limit set} is a closed orbit\index{Closed orbit}. 
\end{prop}

\begin{proof}
If $\omega_X(q)$ is a singularity, then it is done. Hereafter, we assume that $\omega_X(q)$ is not a singularity.
From {\em Lemma \ref{lem6}} follows that $\omega_X(q)$ has singular partitions of arbitrarily small size. 
On the other hand, let $T_U M=\hat{F}_U^s\oplus \hat{F}_U^c$ be a continous extension of the sectional-hyperbolic
splitting \index{Sectional-hyperbolic splitting}
$T_{\omega_X(q)}M=F_{\omega_X(q)}^s\oplus F_{\omega_X(q)}^c$ of $\omega_X(q)$ to a neighborhood $U$ of 
$\omega_X(q)$. Let $I$ be an arc tangent to $\hat{F}_U^c$, transverse to $X$, with $q$ as boundary point.
{\em Theorem 18} in \cite{lec} guarantees for every singular partition \index{Singular partition} 
$\mathcal{R}=\{S_1,\cdots S_k\}$
of $\omega_X(q)$, the existence of $S\in\mathcal{R}$, $\delta>0$, a sequence $q'_1,q'_2,\cdots\in S$ in the positive orbit 
of $q$, and a sequence of intervals $J_1',J_2'\cdots\subset S$ in the positive orbit of $I$ with $q_j'$ as a boundary point 
of $J_j'$ for all such that $length (J_j')\geq\delta$, for all $j=1,2,3,\cdots$.\\

We can assume $I=J_1'$. As $q, q'_j\in M(X)$ and $X$ is
a Venice mask, we can use the {\em Lemma \ref{lem2}} to obtain a sequence $\{q_n:n\in\nat\}\subset M$ such that 
$q_n\to q$ and $\omega(q_n)$ is a singularity for any $n$. As $X$ has just a finite singular points, we can take 
$\omega(q_n)=\{\sigma'\}$ for all $n$, and some $\sigma'\in Sing(X)$. If $q_n\in W^u(\sigma)$ for all $n$,
then $\omega(q)=\{\sigma'\}$ which contradicts our assumption. Therefore $q_n\notin W^ u(\sigma)$ for any $n$. We
can take $q_n$ such that $q_n\in S$ for all $n$\\

On the other hand, for $\sigma'$ are possible the following two alternatives, either $\sigma'\in\omega_X(q)$, or 
$\sigma'\notin\omega_X(q)$. We begin to consider $\sigma'\in\omega_X(q)$. {\em Lemma 14} in \cite{lec} asserts 
$O^+(q)\cap\mathcal{R}=\{\hat{q}_1,\hat{q}_2,\cdots\}$ an infinite sequence ordered in a way that $\Pi(\hat{q}_n)=
\hat{q}_{n+1}$, and the existence of a curve $c_n\subset W^s(Sing(X)\cap \omega_X(q))\cap B_{\delta}(\hat{q}_n)$ 
such that

$$B_{\delta}^+(\hat{q}_n)\subset Dom(\Pi)\qquad and\qquad \Pi |_{B_{\delta}^+(\hat{q}_n)}\quad is \quad C^1, $$

where $B_{\delta}^+(\hat{q}_n)$ denotes the connected component \index{Connected component}
of $B_{\delta}(\hat{q}_n)\setminus c_n$ containing $\hat{q}_n$.

In particular, we can reduce $\delta$ to obtain $\Pi_S=\Pi|_S$  such that

$$(\Pi_S)|_{B_{\delta}^+(q)}\quad\textit{is $C^1$}.$$

However $W^s(\sigma')$ accumulates $q$ on $S$, so we obtain a contradiction.

Therefore the first alternative cannot occur. We conclude $\sigma'\notin \omega_X(q)$.\\

Hartman-Grobman's Theorem \index{Hartman-Grobman's Theorem} implies the existence of a
neighborhood $V_{\sigma'}$ of $\sigma'$, where the flow is $C^0$-conjugated to its linear part. Let $\eta>0$ be such that
$V_{\sigma'}\subset B_{\eta}(\sigma')$ and $O^+(q)\cap V_{\sigma'}=\emptyset$.
From {\em Lemma 2.2} in \cite{mp2} there are singular cross sections $\Sigma^+,\Sigma^-\subset V_{\sigma'}$ such that 
every orbit of $M(X)$ passing close to some point in $W^{s,+}(\sigma')$ (respectively $W^{s,-}(\sigma')$)
intersects $\Sigma^+$(respectively $\Sigma^-$). Moreover {\em Lemma 2.3} in \cite{apu} guarantees the existence of two
disks $\Lambda^+, \Lambda^-\subset V_{\sigma'}$ transverse to $X$ such that for $B_{\varepsilon}(\sigma')\subset V_{\sigma'}$,
and for any point $x\in B_{\varepsilon}(\sigma')$, there are two numbers $t_- < 0 < t_+$
with $X_{t_-}(x) \in \Sigma^+\cup\Sigma^-$ and $X_{t_+}(x)\in \Lambda^+\cup\Lambda^-$. In addition, 
$X_t(x)\in V_{\sigma'}$ for all $t\in(t_-, t_+)$. See Figure \ref{singcross}.\\

\begin{figure}
\begin{center}
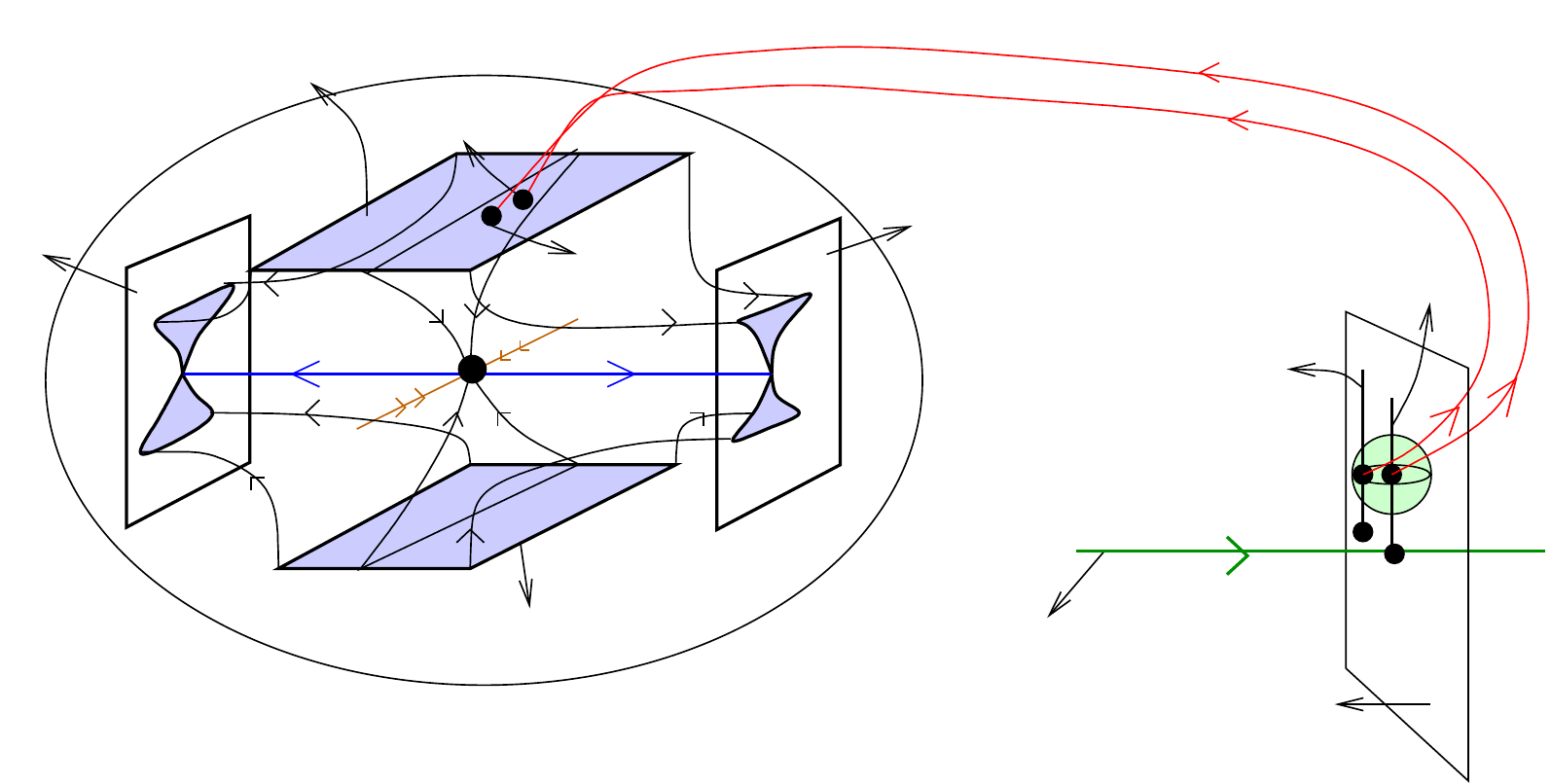
\caption{\label{singcross} Proof {\em Proposition \ref{prop31}.}}
\end{center}
\end{figure}

As $q_n\to q$, we can take a sequence of open arcs $I_1,I_2,\cdots$ with $q_n$ as a boundary point 
of $I_n$ such that $Cl(I_n)$ converges to $Cl(I)$. In particular, we can assume $\delta\leq length (I_n)<\epsilon$
for all $n=1,2,3,\cdots$ and $diam(S)=\epsilon$. In addition, we can take $I_n\subset S$ for all $n$.
On the other hand, $q_n\in W^s(\sigma')$ implies that $O^+(q_n)$ intersects $\Sigma^+\cup\Sigma^-$. Assume that the 
intersection occurs in $\Sigma^+$ for all $n$. As we can choose the singular partition of arbitrarily small size and 
$q$ is non-recurrent, there is $\varepsilon'>0$ such that $diam(\mathcal{R})=\varepsilon'$ and 
$O^+(s_n)\cap \Sigma^+\neq\emptyset$ for all $s_n\in I_n$. \\

We assert that $q$ satisfies the property $(P)_{\Sigma}$, where $\Sigma=\Sigma^+$. Indeed, from $O^+(q)\cap V_{\sigma'}
=\emptyset$ follows $O^+(q)\cap \Sigma^+=\emptyset$. Now, for $x\in I$ there are $\beta_1, \beta_2>0$
such that $B_{\beta_1}(x)\cap \partial I=\emptyset$, $B_{\beta_2}(x)\cap \{q_l\}=\emptyset$ and
$B_{\beta_2}(x)\cap I_l\neq\emptyset$ $l$ for all $l$ large. We define $\beta=\min\{\beta_1,\beta_2\}$.
Let $\{x_l\}_l$ be a sequence with $x_l\in I_l\cap B_{\beta}(x)$ such that $x_l\to x$. As in \cite{lec}, we define
the {\em holonomy map}\index{Holonomy map} $\Pi_{S,\Sigma^+}$ from $S$ to $\Sigma^+$ by 

$$
Dom(\Pi_{S,\Sigma^+})=\{y\in S: X_t(y)\in\Sigma^+\textit{ for some $t>0$}\}
$$

and

$$
\Pi_{S,\Sigma^+}(y)=X_{t_{S,\Sigma^+}(y)}(y),
$$

where $t_{S,\Sigma^+}(y)=\inf\{t>0: X_t(y)\in\Sigma^+\}$.\\ 	

Therefore $x_l\in Dom(\Pi_{S,\Sigma^+})$ for all $n$. From {\em Lemma 19} and {\em Theorem 22} in \cite{lec} follows that
$x\in Dom(\Pi_{S,\Sigma^+})$.\\

Finally, {\em Theorem \ref{thPsigma}} implies that $\omega_X(q)$ is a closed orbit. As we assume $\omega_X(q)$ not being a
singularity, then we conclude that the omega-limit set of $q$ is a periodic orbit.

\end{proof} 

\subsection{Property $(P_{\sigma'})_q^+$}

\begin{defi}
Let $\sigma,\sigma'\in Sing(X)$ and $q$ be a regular point in $W^u(\sigma)$. We say that an open arc $I\subset M$ satisfies 
the Property $(P_{\sigma'})_q^+$ if $q\in \partial I$ and $I\cap W^{s,+}(\sigma')$ is dense in $I$. In a similar way, 
an open arc $J\subset M$ satisfies the Property $(P_{\sigma'})_q^-$ if $q\in \partial J$ and $J\cap W^{s,-}(\sigma')$
is dense in $J$. \index{Property $(P_{\sigma'})_q^+$}
\end{defi}

\begin{prop}
\label{prop41}
Let $O$ be a hyperbolic periodic orbit of a Venice mask $X$. Assume $\sigma'\in Sing(X)$ 
satisfying $\emptyset \neq W^u(O)\cap W^s(\sigma')\subset W^{s,+}(\sigma')$. Let $q$ be
a regular point with $q\in W^u(\sigma)\cap Cl(W^u(O))$, for some $\sigma\in Sing(X)$. Then there is 
an open arc satisfying the Property $(P_{\sigma'})_q^+$. The same interchanging $+$ by $-$.
\end{prop}

\begin{proof}
Let $p\in W^u(\sigma')$ be a regular point. We assert that there is an open interval $J$ satisfying the Property 
$(P_{\sigma'})_p^+$. Indeed, $\sigma'$ and $p$ are contained in $Cl(W^u(O))$. As $W^u(O)$ intersects $W^{s,+}(\sigma')$, 
then $W^u(O)\cap W^{s}(\sigma)$ is dense in $W^{s,+}(\sigma')$. Consider an open arc $J\subset W^u(O)$ with 
$p\in \partial J$. So, the density of $W^u(O)\cap W^{s,+}(\sigma)$ in $W^{u}(O)$ implies that $J\cap W^{s,+}(\sigma')$ 
is dense in $J$. \\

If $\sigma=\sigma'$, then we obtain the desired result. 
Now, we consider $\sigma\neq\sigma'$. From {\em Lemma \ref{prop31}} follows that the omega-limit set of every point
in $W^u(\sigma')$ is a closed orbit. Now, take two point $p_1,p_2$, one on each branch of $W^u(\sigma')\setminus\{\sigma'\}$. 
We analize the following cases which are ilustrated in Figure \ref{omegacases}.

\begin{figure}
\begin{center}
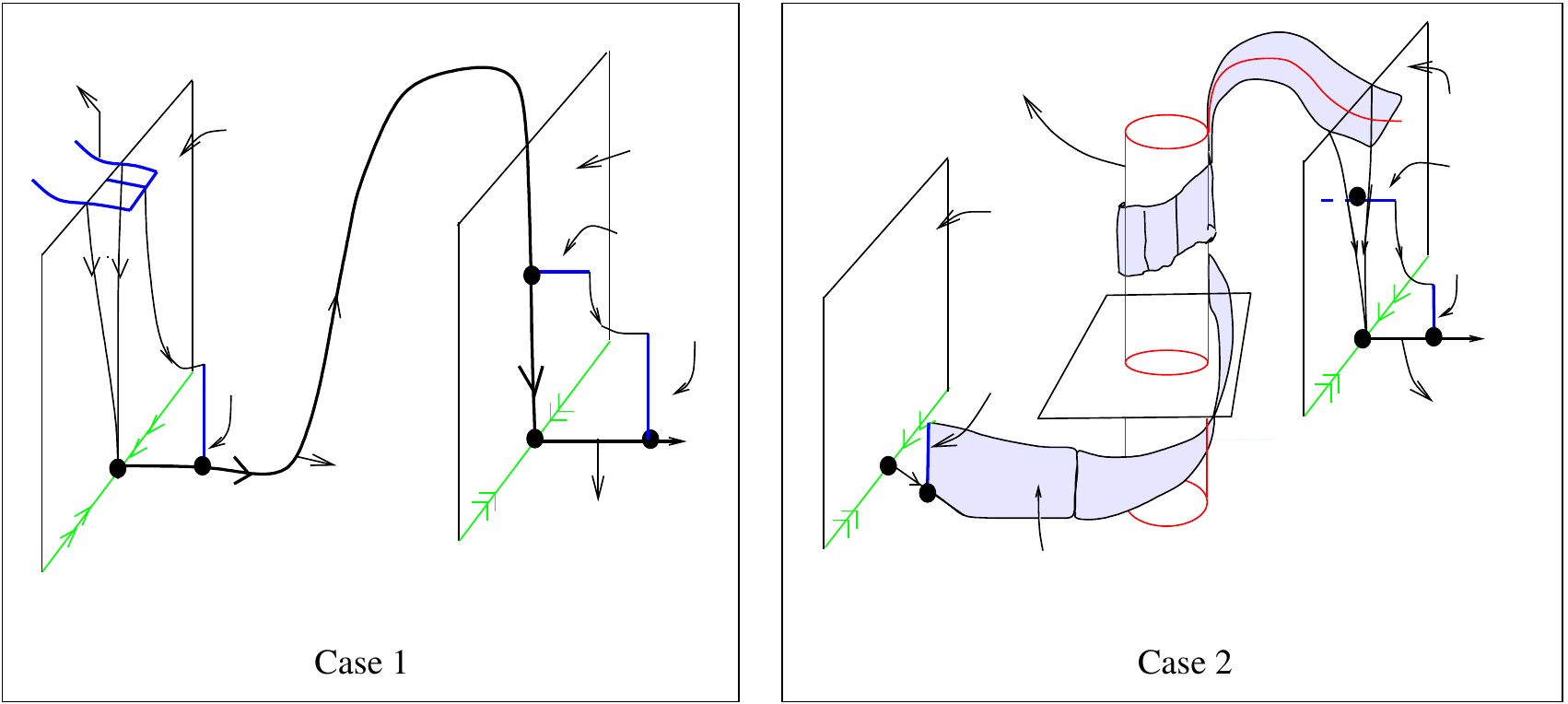
\caption{\label{omegacases} Proof Proposition \ref{prop41}}
\end{center}
\end{figure}

\begin{itemize}

 \item $\omega_X(p_1)$ is a singularity. Let $\sigma_1$ be a singularity with $\omega_X(p_1)=\{\sigma_1\}$. 
If $\omega_X(p_1)=\{\sigma'\}$, then $\omega_X(p_2)\neq\{\sigma'\}$. Indeed, $\omega_X(p_1)=\{\sigma'\}=\omega_X(p_2)$
implies either $W^u(O)\cap W^s(\sigma)\neq\emptyset$ or $Cl(W^u(O))\cap W^s(\sigma)\neq\emptyset$. But
$W^u(O)\cap W^s(\sigma)=\emptyset$ by hypothesis. Moreover $\sigma\in Cl(W^u(O))$. So, $\sigma_1\neq\sigma'$.

Let  $w\in W^u(\sigma')\cap W^s(\sigma_1)$ be a point in $O_X^+(p_1)$ close to $\sigma_1$. Using it and linear
coordinates around $\sigma_1$, we can construct an open interval $J_1\subset\bigcup_{t\geq 0}X_t(J)\subset W^u(O)$
contained in a suitable cross section throught $w$, such that $w\in\partial J_1$. From {\em Inclination lemma} \cite{dmp},
follows that $W^u(O)$ accumulates points in some branch of $W^u(\sigma_1)$. Therefore, for 
$q_1\in (W^u(\sigma_1)\cap Cl(W^u(O)))\setminus\{\sigma_1\}$ there is an open arc $I_1$ such that
$I_1\subset\bigcup_{t\geq 0}X_t(J_1)$ and $q_1\in\partial I_1$. The density of $W^{s,+}(\sigma')\cap 
W^u(O)$ in $W^u(O)$ implies the density of $W^{s,+}(\sigma')\cap I_1$ in $I_1$. Then $I_1$ satisfies 
$(P_{\sigma'})_{q_1}^+$.

 \item When the omega-limit set of $p_1$ and $p_2$ are respectively hyperbolic periodic orbits $O_1,O_2$, we have that 
$W^u(O_i)$ intersects the stable manifold of some singularity $\sigma_i$ of $X$, $i=1,2$. We first assume 
$\sigma_1=\sigma_2=\sigma'$. That intersection cannot just only occurs in $W^s(\sigma')$ because of this would imply
$\sigma\notin Cl(W^u(O_1)\cup W^u(O_2))$ and $Cl(W^u(O))\subset Cl(W^u(O_1)\cup W^u(O_2))$. But $\sigma\in Cl(W^u(O))$
which produces a contradiction. Therefore we can assume that $W^u(O_1)\cap W^s(\sigma_1)\neq\emptyset$ with
$\sigma_1\neq\sigma'$.

Applying {\em Inclination lemma}\index{Inclination lemma}, $Cl(W^u(O))$ and $\bigcup_{t\geq 0}X_t(J)$ intersect
$W^s(\sigma_1)$ transversally. Again,
let $w\in W^u(O)\cap W^s(\sigma)$ be a point in $\bigcup_{t\geq 0}X_t(J)$ close to $\sigma_1$. Using it and linear
coordinates around $\sigma_1$, we can construct an open interval $J_1\subset W^u(O)$ contained in a suitable cross
section throught $w$. $J_1\setminus\{w\}$ is formed by two open arcs $J_1^+,J_1^-\subset  W^u(O)$. Therefore, for 
$q_1\in W^u(\sigma_1)\setminus\{\sigma_1\}$ there is an open arc $I_1$ such that and $q_1\in\partial I_1$ and,  
$I_1\subset\bigcup_{t\geq 0}X_t(J^+)$, or $I_1\subset\bigcup_{t\geq 0}X_t(J^-)$. The density of $W^{s,+}(\sigma')\cap 
W^u(O)$ in $W^{s,+}(\sigma')$ implies the density of $W^{s,+}(\sigma)\cap I_1$ in $I_1$. Then $I_1$ satisfies 
$(P_{\sigma'})_{q_1}^+$. 

\end{itemize}

If $\sigma_1=\sigma$, then the result is obtained. Otherwise, we apply a similar process to $\sigma_1$ to get 
$\sigma_3\in Sing(X)$ with $\sigma_3\notin\{\sigma',\sigma_1\}$, and an open arc $I_3\subset Cl(W^u(O))$ such that
$I_3$ satisfies the Property $(P_{\sigma'})_{q_3}^+$.

As $\sigma\in Cl(W^u(O))$ and $X$ just has finitely many singularities, we conclude the existence of some open arc
satisfying the Property $(P_{\sigma'})_{q}^+$ for $q\in W^u(\sigma)\cap Cl(W^u(O))$.

\end{proof}


\subsection{Proof of Theorem \ref{thH}}

It is sufficient to prove the existence of singular partitions\index{Singular partition} of arbitrarily small size. 

Let $q$ be a regular point in $W^u(\sigma)$, where $\sigma\in Sing(X)$.\\

As $M(X)$ is union of homoclinic classes, there is a hyperbolic periodic orbit $O$ such that $\sigma$ and $q$ are 
contained in the homoclinic class associated to $O$, denoted by $H(O)$. In addition $H(O)$ intersects 
only one or the two connected components $W^{s,+}(\sigma),W^{s,-}(\sigma)$ of $W^s(\sigma)\setminus\mathcal{F}^{ss}_X
(\sigma)$. We begin to analize the intersection in $W^{s,+}(\sigma)$.
On the other hand, $X$ satisfies the Property $(P)$.
This implies that there is a singularity $\sigma'\in Sing(X)$ with $W^u(O)\cap W^s(\sigma')\neq\emptyset$. 
By {\em Theorem} \ref{th1}, the intersection of $W^u(O)$ with $W^s(\sigma')$ is either only one or the two connected 
components $W^{s,+}(\sigma'),W^{s,-}(\sigma')$ of $W^s(\sigma')\setminus\mathcal{F}^{ss}_X(\sigma')$. If $\sigma=\sigma'$
then from {\em Lemma \ref{lem6}} follows the existence of singular partitions of arbitrarily small size.
Hereafter, we assume $\sigma\neq \sigma'$ and $W^{s.+}(\sigma')\cap W^u(O)\neq\emptyset$.\\

If $Cl(W^u(O))\cap W^{s,-}(\sigma')\neq\emptyset$, then {\em Lemma} \ref{lem41} and {\em Proposition} 
\ref{prop31} imply that for some $p\in W^u(\sigma')\cap Cl(W^u(O))$, $O=\omega_X(p)$ and $H(O)\subset Cl(W^u(\sigma'))$.
But $q\notin W^u(\sigma')$. This contradicts $q\in H(O)$. So, $Cl(W^u(O))\cap W^{s,-}(\sigma')=\emptyset$.
{\em Proposition} \ref{prop41} guarantees the existence of
an open arc $I^+\subset M$ satisfying the Property $(P_{\sigma'})_q^+$. \\

We suppose $\omega_X(q)$ is not a periodic orbit. Let $z$ be a point in $\omega_X(q)$. 
In a similar way as {\em Lemma} \ref{lem6}, we fix a foliated rectangle of small diameter $R_z^0$ such that 
$z\in Int(R_z^0)$ and $\omega_X(q)\cap\partial^h R^0_z =\emptyset$. The positive
orbit of $q$ intersects either only one or the two connected components of $R_z^0\setminus\mathcal{F}^s(z, R_z^0)$.
 
Assume the intersection is occurring in just one component only.

Now, analize the following cases:

\begin{itemize}
 
\item  $q\notin H(O')$ for all hyperbolic periodic orbit $O'$ of $X$ such that $H(O')\cap W^{s,-}(\sigma)\neq\emptyset$.

The existence of the singular partitions of arbitrarily small size is obtained such as the first case
in {\em Lemma} \ref{lem6}.

\item There is a sequence $\{p_n\}_n\subset W^u(O)$ such that $p_n\to p\in W^{s,-}(\sigma)$, and there is a 
sequence $\{q_n\}$ such that $q_n\in O_X(p_n)$ and $q_n\to q$.

From {\em Lemma} \ref{lem41} follows that $\omega_X(q)=O$. But this contradicts our assumption that the omega-limit set
is not a periodic orbit. 

\item For some periodic orbit $O'\neq O$, there is a sequence $\{p_n:n\in\nat\}\subset W^u(O')$ such that
$p_n\to p\in W^{s,-}(\sigma)$, and there is a sequence $\{q_n:n\in\nat\}$ satisfying $q_n\in O_X(p_n)$ and $q_n\to q$.

Again, {\em Lemma} \ref{lem41} implies that $W^u(O')$ does not intersect the open arc $I^+$. From Property $(P)$, there is
$\sigma''\in Sing(X)$ such that $W^u(O')\cap W^s(\sigma'')\neq\emptyset$. Then for some $r\in W^u(\sigma'')$ there is 
an interval $J^-\subset W^u(O')$, such that $r\in\partial J^⁻$ and $J^-\cap W^s(\sigma'')$ is dense in $J^-$. Also there is
an open arc $I^-\subset \bigcup_{t\geq 0}X_t(J^-)$ satisfying $q\in\partial I^-$. Therefore $I^-\subset W^u(O')$ and 
$I^-\cap W^s(\sigma'')$ is dense in $I^-$. In addition, $W^{s,+}(\sigma)\cap I^-=\emptyset$. The stable manifolds
throught $I=I^+\cup\{q\}\cup I^-$ generates a subrectangle $R_I$. This rectangle acts such as {\em Lemma} 17 in \cite{lec}.

\end{itemize}

The existence of the singular partition of arbitrarily small size is obtain such as {\em Lemma} \ref{lem6}.\\

If the intersection of $O^+_X(q)$ with $R^0_z$ occurs in both connected components of $R^0_z\setminus\mathcal{F}^s(z,R^0_z)$,
then we proceed such as {\em Lemma} \ref{lem6} to get a cross section $\Sigma_z$ with $z\in\Sigma_z$ and 
$\partial\Sigma_z\cap\omega_X(q)=\emptyset$.

In this way, {\em Proposition} 3 in \cite{lec} implies the existence of the singular partition of arbitrarily 
small size for $\omega_X(q)$.

Finally, we follow the proof of {\em Proposition} \ref{prop31} to conclude that $\omega_X(q)$ is a closed orbit.


\section{Intersection of homoclinic classes}
\label{inthomcla}

In this section we are interested in the study of the intersection of homoclinic classes in a 
sectional-Anosov flow. We follow some ideas developed in \cite{bamo} to obtain {\em Theorem \ref{thH'}}. More
specifically, we prove that in this context, this intersection can be decomposed in three specific sets.
a non-singular hyperbolic set, finitely many singularities and regular orbits joining them.
Recall that an invariant set
is nontrivial if it does not reduces to a single orbit. The conclusion of {\em Theorem \ref{thH'}} is obvious
when $H_1$ or $H_2$ are trivial invariant sets. Hereafter, $H_1$ and $H_2$ are two non trivial different 
homoclinic classes in $M(X)$. Let $\Lambda$ be the intersection between $H_1$ and $H_2$. 
We start with the following lemma.

\begin{lemma}
\label{lem7}
Assume that there is a singularity $\sigma\in\Lambda$, then for $\delta>0$ small, every sequence 
$\{x_n:n\in\nat\}\subset\Lambda\cap B_{\delta}(\sigma)$ 
such that $x_n\to\sigma$ is contained in $W^s(\sigma)\cup W^u(\sigma)$.
\end{lemma}

\begin{proof}
We suppose by contradiction that there is a sequence $\{x_n:n\in\nat\}\subset\Lambda\cap B_{\delta}(\sigma)$ such that
$x_n\to\sigma$ and $x_n\notin W^s(\sigma)\cup 
W^u(\sigma)$ for all $n$. 

So, we obtain two sequences $x^s_n$ and $x_n^u$, in the orbit of $x_n$ such that $x_n^s\to y^s$ 
and $x_n^u\to y^u$ for some $y^s\in W^s(\sigma)\setminus\{\sigma\}$ and $y^u\in W^u(\sigma)\setminus
\{\sigma\}$ close to $\sigma$. Let $O_1,O_2$ be two orbits such that $H(O_1)=H_1$ and $H(O_2)=H_2$. Then
there exist sequences $\{p_n:n\in\nat\}\subset (W^u(O_1)\cap W^s(O_1))$ and $\{q_n:n\in\nat\}\subset 
(W^u(O_2)\cap W^s(O_2))$ satisfying $p_n\to x_n^s$ and 
$q_n\to x_n^s$. We can assume $p_n\notin H_2$ for all $n$.
This means that $p_n\to x^s$ and $q_n\to x^s$ too. The behavior of the orbits of $x_n$, $p_n$ and $q_n$ nearby 
$\sigma$, are as described in Figure \ref{lemmaint}.

Since homoclinic classes have density of periodic points \cite{haka}, for each $n$ we have that $p_n$ and $q_n$ are 
approximated respectively by a sequence of periodic orbits $\{O_1^{mn}:m\in\nat\}$ and $\{O_2^{mn}:m\in\nat\}$. Define 
the map $\pi:B_{\delta}(\sigma)\to W^{cu}(\sigma)$ such as in {\em Subsection \ref{singpart}}.
Observe that $\{\pi(W^u(O_1^{mn})):m\in\nat\}$ and $\{\pi(W^u(O_2^{mn})):m\in\nat\}$ accumulate $y^s$ in the same sector
$s_{ij}$ of $W^{cu}(\sigma)$. 
Follows from {\em Lemma} 3.1 in \cite{carmo} that these sequences can be chosen in a way that, for $i=1,2$ and for all
$n,m$, $W^s(O_i^{nm})$ is uniformly bounded away from zero. This implies that for $m_1,m_2,n_1,n_2$ large,
$W^u(O_1^{m_1n_1})\cap W^s(O_2^{m_2n_2})\neq\emptyset$.
Consider $x\in W^u(O_1^{n_1m_1})\cap W^s(O_2^{m_2n_2})$. As $O_1^{m_1n_1}\subset (H_1\setminus H_2)$ and $O_2^{m_2n_2}
\subset H_2$, then there is $x^*\in O_X(x)$ such that $x^*\in\Lambda$. But $\Lambda$ is an invariant closed set, then
$O_1^{m_1n_1}\subset Cl(O_X(x^*))=Cl(O_X(x^*))\subset \Lambda$. However $O_1^{m_1n_1}\nsubseteq H_2$ and
$\Lambda\subset H_2$, which is a contradiction.

We conclude $x_n\in W^s(\sigma)\cup W^u(\sigma)$ for all $n\in\nat$.

\end{proof}

\begin{figure}
\begin{center}
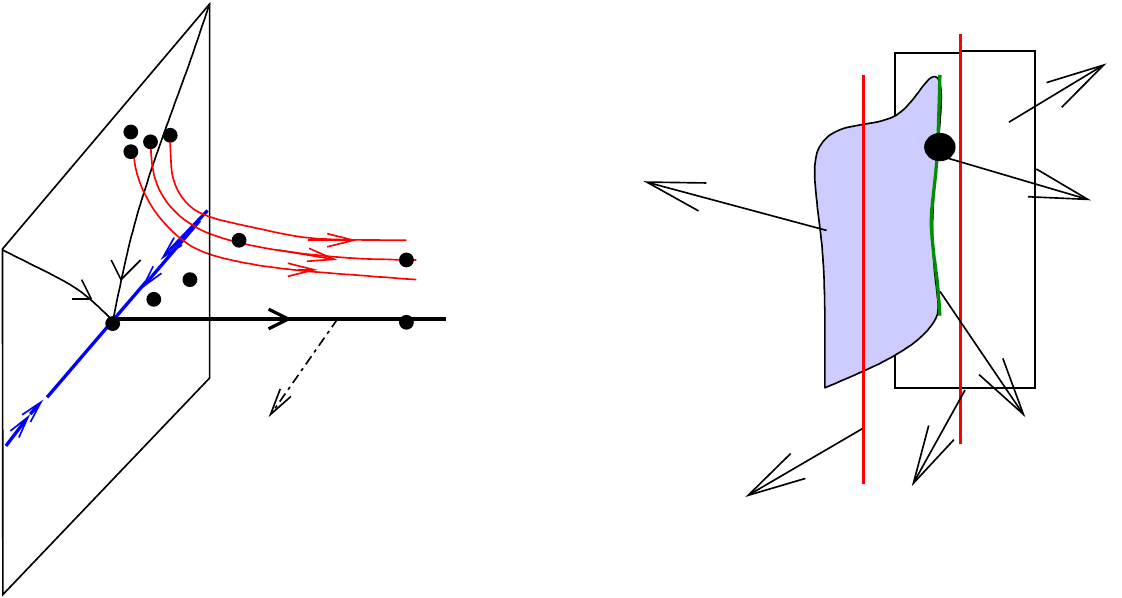
\caption{\label{lemmaint} {\em Lemma} \ref{lem7}}
\end{center}
\end{figure}

\subsection{Proof theorem \ref{thH'}}

Theorem \ref{thH'} gives a description about the set $\Lambda$.

\begin{proof}
The idea of the proof is the same given in {\em Lemma 3.3} by \cite{bamo}. Follows to
{\em Lemma \ref{lem7}} that there is $\delta>0$ such that $\Lambda\cap B_{\delta}(\sigma)\subset 
W^s(\sigma)\cup W^u(\sigma)$, and the balls $B_{\delta}(\sigma)$
are pairwise disjoint for every $\sigma\in\Lambda\cap Sing(X)=S$. Define

$$H=\bigcap_{(t,\sigma)\in \re\times S}X_t(\Lambda\setminus B_{\delta}(\sigma)).$$

By construction, $H$ is a non-singular, \index{Non-singular set} compact invariant sectional-hyperbolic set.
So, applying {\em Lemma \ref{lem1}} we have that $H$ is hyperbolic. Now define $R=\Lambda\setminus(S\cup H)$.
For $x\in R$ there is $(t,\sigma)\in\re\times S$ with $X_t(x)\in B_{\delta}(\sigma)$, and by {\em Lemma
\ref{lem7}} $X_t(x)\in W^s(\sigma)\cup W^u(\sigma)$. 

If $x\in W^u(\sigma)$ we obtain $\alpha(x)\subset H\cup S$. Assume $X_s(x)\notin \bigcup_{\rho\in S}
B_{\delta}(\rho)$ for all $s\geq 0$, then $\omega(x)\subset H$. Now, if there is $(s,\rho)\in\re\times S$
such that $X_s(x)\in B_{\delta}(\rho)$ then $x\in W^s(\rho)$, So $\omega(x)\in H\cup S$. 

With a similar argument we have $\alpha(x)\subset H\cup S$ and $\omega(x)\subset H\cup S$ for
$x\in W^s(\sigma)$. So, we conclude the result.

\end{proof}

\section{Some conjectures}

Because of the study developed in this work, different questions have appeared.  
All known examples of Venice mask are characterized because the maximal invariant set is the finite union of homoclinic
classes and the intersection between two different homoclinic classes $H_1$ and $H_2$ is contained in $Cl(W^u(Sing(X)))$. 
Moreover, every regular point $q\in W^u(Sing(X))\cap H_1\cap H_2$ is non-recurrent.

Consider a Venice mask $X$ supported on a compact 3-manifold $M$.
Let $H_1$ and $H_2$ be two different homoclinic classes in $M(X)$ and let $\Lambda$ be the intersection between
$H_1$ and $H_2$. Assume the decomposition of $\Lambda$ given in {\em Theorem \ref{thH'}}, it is 
$\Lambda=S\cup H\cup R$. \\

We announce the following conjecture.

\begin{conj}
\label{conj1}
Every regular point $q\in R$ is non-recurrent. 
\end{conj}

From {\em Lemma \ref{lem7}} we have that for $\delta>0$ small, $x\in B_{\delta}(\sigma)$ implies
$x\in W^s(\sigma)\cup W^u(\sigma)$ for some $\sigma\in S$. If
$x\in W^u(\sigma)$ then $\alpha(x)=\{\sigma\}$. Now we take $x\in W^s(\sigma)\setminus W^u(\sigma)$. 
Therefore we shall consider two cases, either $\alpha(x)=\{\rho\}$ for some $\rho\in S$ or $\alpha(x)\subset H$.
In the first case, we obtain the desired result. If we prove that the second case cannot occur, then the following
conjecture would be true.

\begin{conj}
\label{conj2}
$\Lambda\subset Cl(W^u(Sing(X)))$. 
\end{conj}

Let us state direct consequence of the hyperbolic Lemma \ref{lem1} that appears in \cite{lec}.

\begin{clly}
Every periodic orbit of a sectional-Anosov flow on a compact manifold is hyperbolic. 
In particular, all such flows have countably many closed orbits.
\end{clly}

This implies that the maximal invariant set of every Venice mask is union of countably many homoclinic classes. 
So, if {\em Conjecture \ref{conj1}} and  {\em Conjecture \ref{conj2}} are true, then would be possible to realize the
following statement.

\begin{conj}
\label{conj3}
The maximal invariant set of every Venice mask is finite union of homoclinic classes. 
\end{conj}

\begin{proof}
Let $X$ be a Venice mask supported on a compact 3-manifold $M$. Then $X$ has finite many singularities, we say $n$. 
Let $H_1$, $H_2$ be two different homoclinic classes associated to $M(X)$. From Conjectures \ref{conj1} and \ref{conj2}
is possible to apply Theorem \ref{thH} to conclude that for each singularity
$\sigma$ of $X$, $Cl(W^u(\sigma))=\{\sigma\}\cup W^u(\sigma)\cup C_{\sigma}$, it is a disjoint union and 
$C_{\sigma}$ is a closed orbit. On the other hand, the branches of $W^u(\sigma)$ are uni-dimensional. 
Therefore Theorem \ref{conj2} implies $H_1\cap H_2$ has just only a finite number of possibilities to occur. Moreover, at
most three homoclinic classes can contain the branch of the unstable manifold of some singularity.  

This finishes the proof.

\end{proof}

\bibliographystyle{acm}

\bibliography{RCMBibTeX}

\medskip

\flushleft
H. M. S\'anchez\\
Instituto de Matem\'atica, Universidade Federal do Rio de Janeiro\\
Rio de Janeiro, Brazil\\
E-mail: hmsanchezs@unal.edu.co

\end{document}